\documentclass[11pt,reqno]{amsart}
\usepackage{graphicx}
\usepackage{amsmath,amssymb,mathrsfs}
\usepackage[usenames]{color}

\newtheorem{thm}{Theorem}[section]

\newtheorem{lem}{Lemma}[section]
\newtheorem{prop}{Proposition}[section]
\theoremstyle{definition}

\theoremstyle{remark}
\newtheorem{rem}{Remark}[section]
\numberwithin{equation}{section}

\begin{document}

\title[Nearly Cloaking Full Maxwell's Equations]{Nearly Cloaking the Full Maxwell Equations}

\author{Gang Bao}
\address{Department of Mathematics, Zhejiang University, Hangzhou 310027,
China; Department of Mathematics, Michigan
State University, East Lansing, MI 48824.}
\email{bao@math.msu.edu}

\author{Hongyu Liu}
\address{Department of
Mathematics and Statistics, University of North Carolina, Charlotte,
NC 28223.}
\email{hongyu.liuip@gmail.com}



\date{}

\maketitle

\begin{abstract}
The approximate cloaking is investigated for time-harmonic Maxwell's equations via the approach of transformation optics. The problem is reduced to certain boundary effect estimates due to an inhomogeneous electromagnetic inclusion with an asymptotically small support but an arbitrary content enclosed by a thin high-conducting layer. Sharp estimates are established in terms of the asymptotic parameter, which are independent of the material tensors of the small electromagnetic inclusion.
The result implies that the `blow-up-a-small-region' construction via the transformation optics approach yields a near-cloak for the electromagnetic waves. A novelty lies in the fact that the geometry of the cloaking construction of this work can be very general. Moreover, by incorporating the conducting layer developed in the present paper right between the cloaked region and the cloaking region, arbitrary electromagnetic contents can be nearly cloaked. Our mathematical technique extends the general one developed in \cite{LiuSun} for nearly cloaking scalar optics. In order to investigate the approximate electromagnetic cloaking for general geometries with arbitrary cloaked contents,  new techniques and analysis tools must be developed for this more challenging vector optics case.
\end{abstract}

\section{Introduction and statement of the main result}

This paper is concerned with invisibility cloaking for electromagnetic (EM) waves via the approach of transformation optics \cite{GLU,GLU2,Leo,PenSchSmi}, which is a rapidly growing scientific field with many potential applications. We refer to  \cite{CC,GKLU4,GKLU5,U2,YYQ} and the references therein
for discussions of the recent progress on both the theory and experiments.

Let $D$ and $\Omega$ be two bounded simply connected smooth domains in $\mathbb{R}^3$ such that $D\Subset\Omega$ and $D$ contains the origin. Denote
\[
D_\rho:=\{\rho x; \ x\in D\}\quad \mbox{for\ $\rho\in\mathbb{R}_+$}.
\]
Let $\varepsilon(x)=(\varepsilon^{ij}(x))_{i,j=1}^3$, $\mu(x)=(\mu^{ij}(x))_{i,j=1}^3$ and $\sigma(x)=(\sigma^{ij}(x))_{i,j=1}^3$, $x\in\Omega$ be real symmetric-matrix-valued functions, which are bounded in the sense that
\begin{equation}\label{eq:uniform elliptic}
c|\xi|^2\leq \sum_{i,j=1}^3 \varepsilon^{ij}(x)\xi_i\xi_j\leq C |\xi|^2,\quad c|\xi|^2\leq \sum_{i,j=1}^3 \mu^{ij}(x)\xi_i\xi_j\leq C |\xi|^2
\end{equation}
and
\begin{equation}\label{eq:uniform elliptic 2}
0\leq \sum_{i,j=1}^3 \sigma^{ij}(x)\xi_i\xi_j\leq C |\xi|^2,
\end{equation}
for all $x\in\Omega$ and $\xi=(\xi_i)_{i=1}^3\in\mathbb{R}^3$. Here $c$ and $C$ are two generic positive constants whose meanings should be clear from the contexts. Physically, the functions $\varepsilon$, $\mu$ and $\sigma$ respectively stand for the electric permittivity, magnetic permeability and conductivity tensors of a {\it regular} EM medium occupying $\Omega$.

Let $0<\rho<1$ be a small parameter. Assume that there exists an orientation-preserving bi-Lipschitz mapping $F_\rho: \overline{\Omega}\backslash D_\rho\rightarrow \overline{\Omega}\backslash D$, such that
\begin{equation}\label{eq:map1}
F_\rho(\overline{\Omega}\backslash D_\rho)= \overline{\Omega}\backslash D,\qquad F_\rho|_{\partial\Omega}=\mbox{Identity}.
\end{equation}
Set
\begin{equation}\label{eq:map whole}
F(x)=\begin{cases}
F_\rho(x),\quad & x\in \Omega\backslash\overline{D}_\rho,\\
\frac{x}{\rho},\quad & x\in D_\rho.
\end{cases}
\end{equation}
Define an EM medium inside $\Omega\backslash\overline{D}$ as follows
\begin{equation}\label{eq:cloaking medium}
\varepsilon_c^\rho(x)=F_* \varepsilon_0(x),\quad \mu_c^\rho(x)=F_* \mu_0(x),\quad \sigma_c^\rho(x)=0,\quad x\in\Omega\backslash \overline{D},
\end{equation}
where $\varepsilon_0^{ij}=\delta^{ij}$ and $\mu_0^{ij}=\delta^{ij}$ with $\delta^{ij}$ the Kronecker delta function denote the EM parameter tensors of the homogeneous free space. The {\it push-forward} in \eqref{eq:cloaking medium} is defined by
\begin{equation}\label{eq:pushforward}
F_*m(x):=\frac{DF(y)\cdot m(y)\cdot DF(y)^T}{\left|\mbox{det} (DF)(y) \right|}\bigg|_{y=F^{-1}(x)},\quad x\in\Omega\backslash\overline{D}
\end{equation}
where $m(y), y\in\Omega\backslash\overline{D}_\rho$, denotes an EM parameter in $\Omega\backslash\overline{D}_\rho$, such as $\varepsilon, \mu$ or $\sigma$, and $DF$ represents the Jacobian matrix of the transformation $F$. Also, one may rewrite \eqref{eq:cloaking medium} as
\[
(\Omega\backslash\overline{D}; \varepsilon_c^\rho, \mu_c^\rho)=F_*(\Omega\backslash\overline{D}_\rho; \varepsilon_0, \mu_0)=(F(\Omega\backslash\overline{D}_\rho); F_*\varepsilon_0, F_*\mu_0).
\]
Similarly, set
\begin{equation}\label{eq:conducting layer}
(D\backslash\overline{D}_{1/2}; \varepsilon_l, \mu_l, \sigma_l)=F_*(D_\rho\backslash\overline{D}_{\rho/2}; \alpha_0\varepsilon_0, \beta_0\mu_0, \gamma_0\rho^{-2}\delta ),
\end{equation}
where $\alpha_0,\beta_0$ and $\gamma_0$ are positive constants, and
\begin{equation}\label{eq:target medium}
(D_{1/2}; \tilde{\varepsilon}_a, \tilde{\mu}_a, \tilde{\sigma}_a)=F_*(D_{\rho/2}; \varepsilon_a, \mu_a, \sigma_a),
\end{equation}
with {\it arbitrary} but {\it regular} $\varepsilon_a, \mu_a$ and $\sigma_a$ .
Hence, we have an EM medium in $\Omega$ given by
\begin{equation}\label{eq:EM medium whole}
\Omega; \tilde{\varepsilon}, \tilde{\mu}, \tilde{\sigma}=\begin{cases}
\varepsilon_c^\rho, \mu_c^\rho, \sigma_c^\rho\qquad & \mbox{in\ \ $\Omega\backslash \overline{D}$},\\
\varepsilon_l, \mu_l, \sigma_l\quad & \mbox{in\ \  $D\backslash\overline{D}_{1/2}$},\\
\tilde{\varepsilon}_a, \tilde{\mu}_a, \tilde{\sigma}_a\quad & \mbox{in\ \ $D_{1/2}$}.
\end{cases}
\end{equation}
The time-harmonic EM waves propagating in $\Omega$ is governed by the following Maxwell equations
\begin{equation}\label{eq:Maxwell}
\begin{cases}
& \displaystyle{\nabla\wedge \widetilde{E}_\rho-i\omega \tilde{\mu} \widetilde{H}_\rho=0}\\
&\displaystyle{ \nabla \wedge \widetilde{H}_\rho+i\omega (\tilde{\varepsilon}+i\frac{\tilde{\sigma}}{\omega}) \widetilde{E}_\rho=0}
\end{cases}\quad \mbox{in\ \ $\Omega$},
\end{equation}
where $\widetilde{E}_\rho\in\mathbb{C}^3$ and $\widetilde{H}_\rho\in \mathbb{C}^3$ denote, respectively, the electric and magnetic fields, and $\omega\in\mathbb{R}_+$ denotes the frequency.

Introduce the boundary operator $\widetilde{\Lambda}_\rho$ which maps the tangential component of $\widetilde{E}_\rho|_{\partial\Omega}$ to that of $\widetilde{H}_\rho|_{\partial\Omega}$, i.e.,
\begin{equation}\label{eq:admittance}
\widetilde{\Lambda}_\rho(\nu\wedge\widetilde{E}_\rho|_{\partial\Omega})=\nu\wedge \widetilde{H}_\rho |_{\partial\Omega}: TH_{\text{Div}}^{-1/2}(\partial\Omega)\rightarrow TH_{\text{Div}}^{-1/2}(\partial\Omega),
\end{equation}
where $\nu$ denotes the outward unit normal vector to $\partial\Omega$, and
\[
TH_{\text{Div}}^{-1/2}(\partial\Omega)=\left\{U\in TH^{-1/2}(\partial\Omega)|\ \text{Div}(U)\in H^{-1/2}(\partial\Omega)\right\}\;,
\]
with $\text{Div}$ the surface divergence operator on $\partial\Omega$, $TH^s(\partial\Omega)$ the subspace of all those $V\in (H^s(\partial\Omega))^3$ which are orthogonal to $\nu$ and $H^s(\cdot)$ the usual $L^2$-based Sobolev space of order $s\in\mathbb{R}$. Note that if $\Gamma$ is the smooth boundary of a bounded domain in $\mathbb{R}^3$, then $H^s(\Gamma)$ and hence $TH^s(\Gamma)$ is well defined for $|s|\leq 2$; see \cite{Gri} and \cite{Lio}.
In \eqref{eq:admittance},  $\widetilde{E}_\rho\in H(\nabla\wedge; \Omega)$ is the unique solution to the Maxwell equations \eqref{eq:Maxwell} associated with the following boundary condition
\begin{equation}\label{eq:bc physical}
\nu\wedge \widetilde{E}_\rho|_{\partial\Omega}=\psi\in TH_{\text{Div}}^{-1/2}(\partial\Omega)\;,
\end{equation}
where
\[H(\nabla\wedge; \Omega)=\{ U\in (L^2(\Omega))^3| \nabla\wedge U\in (L^2(\Omega))^3 \}.
\]
In fact, $\widetilde{\Lambda}_\rho$ is also known as the {\it admittance map} in the literature.

We further introduce the `free-space' admittance map as follows. Let $E_0\in H(\nabla\wedge; \Omega)$ and $H_0\in H(\nabla\wedge; \Omega)$ be solutions to
\begin{equation}\label{eq:Maxwell free}
\begin{cases}
& \nabla\wedge E_0-i\omega \mu_0 H_0=0\qquad \, \mbox{in\ \ $\Omega$},\\
& \nabla\wedge H_0+i\omega \varepsilon_0 E_0=0\qquad\ \mbox{in\ \ $\Omega$},\\
& \nu\wedge E_0|_{\partial\Omega}=\psi \in TH_{\text{Div}}^{-1/2}(\partial\Omega),
\end{cases}
\end{equation}
It is assumed that $\omega$ is not an EM eigenvalue to the Maxwell equations \eqref{eq:Maxwell free}; namely, if $\psi=0$, then one must have $E_0=H_0=0$ for \eqref{eq:Maxwell free}. Hence we have a well-defined admittance map
\begin{equation}\label{eq:free admittance}
\Lambda_0(\psi)=\nu\wedge H_0|_{\partial\Omega}: TH_{\text{Div}}^{-1/2}(\partial\Omega)\rightarrow TH_{\text{Div}}^{-1/2}(\partial\Omega),
\end{equation}
where $H_0\in H(\nabla\wedge; \Omega)$ is the unique solution to \eqref{eq:Maxwell free}.
We refer to \cite{Ned} and \cite{Lei} for studies on the well-posedness of the Maxwell equations in the function setting introduced above.

We are ready to state the main result of this paper.

\begin{thm}\label{thm:main}
Suppose $\omega$ is not an EM eigenvalue of the free-space Maxwell equations \eqref{eq:Maxwell free}. Let $\widetilde{\Lambda}_\rho$ be the boundary admittance map in \eqref{eq:admittance} associated with \eqref{eq:Maxwell}, where the EM parameter tensors are given by \eqref{eq:cloaking medium}--\eqref{eq:EM medium whole}. Let $\Lambda_0$ be the ``free" admittance map in \eqref{eq:free admittance} associated with \eqref{eq:Maxwell free}. Then there exists a positive constant $\rho_0$ such that for any $\rho<\rho_0$,
\begin{equation}\label{eq:main estimate}
\|\widetilde{\Lambda}_\rho-\Lambda_0\|_{\mathcal{L}(TH_{\text{\em Div}}^{-1/2}(\partial\Omega),  TH_{\text{\em Div}}^{-1/2}(\partial\Omega) )}\leq C \rho^3,
\end{equation}
where $C$ is a positive constant dependent only on $\rho_0, \omega, \alpha_0,\beta_0, \gamma_0$ and $D$, $\Omega$, but completely independent of $\rho$, $\tilde{\varepsilon}_a$, $\tilde{\mu}_a$ and $\tilde{\sigma}_a$.
\end{thm}

Before we proceed to prove the result, some general remarks about the significance of the result are
in order.

Theorem~\ref{thm:main} states that the transformation medium $(\Omega\backslash\overline{D}; \varepsilon_c^\rho, \mu_c^\rho, \sigma_c^\rho)$ together with the conducting layer $(D\backslash\overline{D}_{1/2}; \varepsilon_l,\mu_l, \sigma_l)$ in \eqref{eq:EM medium whole} produces an approximate invisibility cloaking device which nearly cloaks an arbitrary target medium $(D_{1/2}; \tilde{\varepsilon}_a, \tilde{\mu}_a, \tilde{\sigma}_a)$ located in the innermost region. Indeed, in the limiting case with $\rho=0$ within spherical geometry, namely $\Omega$ and $D$ are both Euclidean balls, \eqref{eq:EM medium whole} without the conducting layer yields the perfect invisibility cloaking construction in \cite{PenSchSmi,GKLU3}. That is, $\Lambda_\rho=\Lambda_0$ for $\rho=0$, and hence by using the exterior boundary measurements encoded in $\Lambda_\rho$, one cannot ``see" the inside object. However, it is widely known that the construction employs singular materials; that is, the material tensors $\varepsilon_c^\rho$ and $\mu_c^\rho$ in the limiting case $\rho=0$ possess degenerate singularities (cf. \cite{GKLU3}). This presents a great challenge for both theoretical analysis and practical fabrications. In order to avoid the singular structure/materials, several regularized constructions have been developed. In \cite{GKLUoe,GKLU_2,RYNQ}, a truncation of singularities has been introduced. In \cite{KOVW,KSVW,Liu}, the `blow-up-a-point' transformation in \cite{GLU2,Leo,PenSchSmi} has been regularized to become the `blow-up-a-small-region' transformation. Nevertheless, as pointed out in \cite{KocLiuSunUhl}, the truncation-of-singularity construction and the blow-up-a-small-region construction are equivalent to each other. Hence, our present study focuses on the blow-up-a-small-region construction; that is, $F_\rho$ is used to blow up $D_\rho$ of the relative size $\rho<\hspace*{-1mm}<1$ for constructing the cloaking medium in \eqref{eq:cloaking medium}. Theorem~\ref{thm:main} states that our cloaking construction \eqref{eq:EM medium whole} yields an approximate cloaking device within $\rho^3$-accuracy of the perfect cloak. Furthermore, an arbitrary content can be nearly cloaked.

Due to its practical importance, the approximate cloaking has recently been extensively studied. In \cite{KSVW,Ammari1}, approximate cloaking schemes were developed for EIT (electric impedance tomography) which might be regarded as optics at zero frequency. In \cite{Ammari2,Ammari3,KOVW,LiLiuSun,LiuSun,Liu}, various near-cloaking schemes were presented for scalar waves governed by the Helmholtz equation. In \cite{LiuZhou}, a similar construction to \eqref{eq:EM medium whole} was developed for the full Maxwell equations. However, the study in \cite{LiuZhou} was only conducted for spherical geometry and the uniform cloaked content; that is, both $\Omega$ and $D$ in \eqref{eq:EM medium whole} were assumed to be Euclidean balls and the medium parameters $\tilde{\varepsilon}_a, \tilde{\mu}_a$ and $\tilde{\sigma}_a$ were all constants multiple of the identity matrix.  Under these assumptions, the Fourier-Bessel technique can be used to derive the analytic series expansions of the EM fields \cite{LiuZhou}.
To our best knowledge, Theorem~\ref{thm:main} is the first result for nearly cloaking the full Maxwell equations with general geometry and arbitrary cloaked contents. In order to assess the near-cloaking construction, the study is shown to be reduced to the boundary effect estimate due to an inhomogeneous electromagnetic inclusion with an asymptotically small support but an arbitrary content enclosed by a thin conducting layer with an asymptotically high conductivity tensor. The new structures of our problem require novel mathematical arguments. Our mathematical analysis uses the general strategy developed in \cite{LiuSun} for nearly cloaking the scalar Helmholtz equation. However, new technique and estimates must be developed to deal with the general vector Maxwell equations.  Finally, we point out that incorporating a damping mechanism by a conducting layer into the near-cloaking construction \eqref{eq:EM medium whole} is necessary for achieving successful near-cloak. In fact, it has been shown in \cite{LiuZhou} that no matter how small the regularization parameter $\rho$ is, there always exist cloak-busting inclusions. Moreover, the result in \cite{LiuZhou} for the special case within spherical geometry and uniform cloaked contents confirms that our estimate in Theorem~\ref{thm:main} is sharp.

For noninvasive EM detections, an related inverse problem is to extract physical information of the interior object, namely, $\tilde{\varepsilon}, \tilde{\mu}$ and $\tilde{\sigma}$ from the knowledge of the exterior EM measurements encoded into the boundary operator $\widetilde{\Lambda}_\rho$.
We refer the reader to \cite{18} and \cite{19} and the references therein for results on uniqueness and stability of this important inverse problem.

The rest of the paper is organized as follows. In Section 2, we present the proof of Theorem~\ref{thm:main}. Section 3 is devoted to the proof of a key lemma that was needed in the proof of Theorem~\ref{thm:main}.

\section{Proof of the main theorem}



%

\subsection{Proof of Theorem~\ref{thm:main}}

We first present a lemma with some key ingredients of the transformation optics, the proofs of which are available in \cite{LiuZhou}.

\begin{lem}\label{thm:trans opt}
Suppose that $E\in H(\nabla\wedge;\Omega)$ and $H\in H(\nabla\wedge; \Omega)$ are EM fileds satisfying
\[
\begin{split}
\nabla\wedge E-i\omega \mu H=0\qquad & \mbox{in\ \ $\Omega$,}\\
\nabla\wedge H+i\omega \left( \varepsilon+i\frac{\sigma}{\omega} \right) E=0\qquad & \mbox{in\ \ $\Omega$},
\end{split}
\]
where $(\Omega; \varepsilon, \mu, \sigma)$ is a regular EM medium. Let $x'=\mathcal{F}(x): \Omega\rightarrow\Omega$ be a bi-Lipschitz and orientation-preserving mapping such that $\mathcal{F}|_{\partial\Omega}=\mbox{Identity}$. Define the {\emph pull-back EM fields} by
\[
\begin{split}
E'&=(\mathcal{F}^{-1})^* E:=(D\mathcal{F})^{-T}E\circ \mathcal{F}^{-1},\\
H'&=(\mathcal{F}^{-1})^* H:=(D\mathcal{F})^{-T}H\circ \mathcal{F}^{-1}.
\end{split}
\]
Then, $E'\in H(\nabla'\wedge;\Omega)$ and $H' \in H(\nabla'\wedge; \Omega)$. In addition
 the following identities hold
\[
\begin{split}
\nabla' \wedge E'=i\omega \mu' H'\qquad & \mbox{in\ \ $\Omega$}\\
\nabla' \wedge H'=-i\omega \left(\varepsilon'+i\frac{\sigma'}{\omega}\right) E'\quad & \mbox{in\ \ $\Omega$},
\end{split}
\]
where $\nabla'\wedge$ denote the {\em curl} operator in the $x'$-coordinates, and $\varepsilon'$, $\mu'$ and $\sigma'$ are the push-forwards of $\varepsilon, \mu$ and $\sigma$ via $\mathcal{F}$, namely,
\[
(\Omega; \varepsilon',\mu',\sigma')=\mathcal{F}_*(\Omega; \varepsilon, \mu, \sigma).
\]
Particularly, if one lets $\Lambda$ and $\Lambda'$ denote the admittance maps associated with $(E, H)$ and $(E', H')$, respectively. Then
\[
\Lambda=\Lambda'.
\]
\end{lem}

Next, for the EM fields $(\widetilde{E}_\rho, \widetilde{H}_\rho)$ in \eqref{eq:Maxwell} associated with the boundary condition \eqref{eq:bc physical}, we let
\begin{equation}\label{eq:EM virtual}
E_\rho=F^* \widetilde{E}_\rho\quad \mbox{and}\quad H_\rho=F^*\widetilde{H}_\rho.
\end{equation}
Then by Lemma~\ref{thm:trans opt}, it is seen that $E_\rho\in H(\nabla\wedge; \Omega)$ and $H_\rho\in H(\nabla\wedge; \Omega)$ which satisfy the following Maxwell equations,
\begin{equation}\label{eq:Maxwell virtual}
\begin{cases}
\displaystyle{\nabla\wedge E_\rho-i\omega \mu_\rho H_\rho=0}\qquad & \mbox{in\ \ $\Omega$},\\
\displaystyle{\nabla\wedge H_\rho+i\omega \left(\varepsilon_\rho+i\frac{\sigma_\rho}{\omega}\right)E_\rho=0} \qquad &\mbox{in\ \ $\Omega$},\\
\displaystyle{\nu\wedge E_\rho|_{\partial\Omega}=\psi\in TH_{\text{Div}}^{-1/2}(\partial\Omega),}
\end{cases}
\end{equation}
where
\begin{equation}\label{eq:virtual medium}
\Omega; \varepsilon_\rho, \mu_\rho, \sigma_\rho=\begin{cases}
\varepsilon_0, \mu_0, 0\qquad &\mbox{in\ \ $\Omega\backslash D_\rho$},\\
\alpha_0\varepsilon_0, \beta_0\rho^2 \mu_0, \gamma_0\rho^{-2}\delta\quad &\mbox{in\ \ $D_\rho\backslash\overline{D}_{\rho/2}$},\\
\varepsilon_a, \mu_a, \sigma_a\qquad &\mbox{in\ \ $D_{\rho/2}$}.
\end{cases}
\end{equation}
Furthermore, by Lemma~\eqref{thm:trans opt},
\begin{equation}\label{eq:virtual map}
\Lambda_\rho=\widetilde{\Lambda}_\rho\;
\end{equation}
where
$\Lambda_\rho$ is the admittance map associated with the EM fileds $(E_\rho, H_\rho)$ in \eqref{eq:Maxwell virtual}.
Hence, in order to prove Theorem~\ref{thm:main}, it suffices to show the following result.

\begin{thm}\label{thm:main virtual}
Suppose $\omega$ is not an EM eigenvalue of the free-space Maxwell equations \eqref{eq:Maxwell free}. Let $(E_0, H_0)\in  H(\nabla\wedge;\Omega)\wedge H(\nabla\wedge;\Omega)$ and $(E_\rho, H_\rho)\in H(\nabla\wedge;\Omega)\wedge H(\nabla\wedge;\Omega)$ be solutions to \eqref{eq:Maxwell free} and \eqref{eq:Maxwell virtual}, respectively. Then there exists a positive constant $\rho_0$ such that for any $\rho<\rho_0$,
\begin{equation}\label{eq:main estimate}
\| \nu\wedge H_\rho-\nu\wedge H_0 \|_{TH^{-1/2}_{\text{\em Div}}(\partial\Omega)} \leq C \rho^3\|\psi\|_{TH^{-1/2}_{\text{\em Div}}(\partial\Omega)},
\end{equation}
where $C$ is a positive constant dependent only on $\rho_0, \omega, \alpha_0,\beta_0, \gamma_0$ and $D$, $\Omega$, but completely independent of $\rho$, ${\varepsilon}_a$, ${\mu}_a$ and ${\sigma}_a$ and $\psi$.
\end{thm}

\begin{rem}
Qualitatively speaking, Theorem~\ref{thm:main virtual} states that the boundary EM effects due to an inhomogeneous EM inclusion supported in a small region $D_\rho$ with arbitrary contents in $D_{\rho/2}$ but enclosed by a thin layer of high conducting medium in $D_\rho\backslash\overline{D}_{\rho/2}$ is also small. In the literature, there are extensive studies on the scattering estimates due to small EM scatterers (see \cite{Ammari4,AmmVogVol,APRT}), and also the closely related low-frequency asymptotics of EM scattering (i.e., the Rayleigh approximation; see \cite{AmmNed,Kle,Mar,Ned}). However, the EM inclusions of the aforementioned studies have fixed contents, and indeed they are either perfectly conducting obstacles or EM mediums with piecewise constant material parameters. In Theorem~\ref{thm:main virtual}, the small EM inclusion of our current study has peculiar structures, and the quantative estimate \eqref{eq:main estimate} cannot be adapted from existing results in the literature.
\end{rem}

The rest of the paper is devoted to the proof of Theorem~\ref{thm:main virtual}. In order to simplify the exposition, we set
\[
\alpha_0=\beta_0=\gamma_0=1.
\]
We next derive three key lemmas.

\begin{lem}\label{lem:1}
The solutions of \eqref{eq:Maxwell free} and \eqref{eq:Maxwell virtual} satisfy
\begin{equation}\label{eq:estimate 1}
\int_{D_\rho\backslash D_{\rho/2}}|E_\rho|^2\ d\sigma_x\leq C \rho^2\|\psi\|_{TH_{\text{\em Div}}^{-1/2}(\partial\Omega)}\|\nu\wedge(H_\rho-H_0)\|_{TH^{-1/2}_{\text{\em Div}}(\partial\Omega)},
\end{equation}
where $C$ is a positive constant depending only on $\Omega$.
\end{lem}

\begin{proof}
Inner-producting both sides of the second equation in \eqref{eq:Maxwell virtual} by $\overline{E}_\rho$, and integrating by parts, we have
\begin{equation}\label{eq:lem1 1}
\begin{split}
&\int_{\Omega}-i\omega\left(\varepsilon_\rho+i\frac{\sigma_\rho}{\omega}\right) E_\rho\cdot\overline{E}_\rho \ d\sigma_x=\int_{\Omega} (\nabla\wedge H_\rho)\cdot \overline{E}_\rho\ d\sigma_x\\
=& \int_{\Omega} H_\rho\cdot(\nabla\wedge\overline{E}_\rho)\ d\sigma_x-\int_{\partial\Omega} (\nu\wedge\overline{E}_\rho)\cdot H_\rho\ d s_x\\
=&\int_{\Omega} H_\rho\cdot(-i\omega\overline{\mu_\rho}\overline{H}_\rho)\ d\sigma_x+\int_{\partial\Omega} (\nu\wedge\overline{E}_\rho)\cdot [\nu\wedge(\nu\wedge H_\rho)]\ ds_x
\end{split}
\end{equation}
In \eqref{eq:lem1 1}, we have made use of the decomposition that
\[
H_\rho|_{\partial\Omega}=(H_\rho|_{\partial\Omega})_t+\nu (H_\rho|_{\partial\Omega})_\nu,
\]
where the tangential component $(H_\rho|_{\partial\Omega})_t$ is $-\nu\wedge(\nu\wedge (H_\rho|_{\partial\Omega}))$ and the normal component is $(H_\rho|_{\partial\Omega})_\nu$ is $\langle\nu, H_\rho|_{\partial\Omega} \rangle$. 
By taking the real parts of both sides of \eqref{eq:lem1 1}, we have
\begin{equation}\label{eq:lem1 2}
\begin{split}
&\int_{D_{\rho/2}} \sigma_a E_\rho\cdot\overline{E}_\rho\ d\sigma_x+ \rho^{-2}\int_{D_\rho\backslash D_{\rho/2}} |E_\rho|^2\ d\sigma_x\\
=&\Re\int_{\partial\Omega}(\nu\wedge \overline{E}_\rho)\cdot [\nu\wedge(\nu\wedge H_\rho)]\ ds_x.
\end{split}
\end{equation}
On the other hand, it is straightforward to verify that
\begin{equation}\label{eq:lem1 3}
0=\Re \int_{\partial\Omega} (\nu\wedge\overline{E}_0)\cdot [\nu\wedge (\nu\wedge H_0)]\ ds_x.
\end{equation}
We shall make use of the following fact that the skew-symmetric bilinear form
\begin{equation}\label{eq:duality}
\begin{split}
&\mathcal{B}: TH_{\text{Div}}^{-1/2}(\partial\Omega)\wedge TH_{\text{Div}}^{-1/2}(\partial\Omega)\ \rightarrow\quad  \mathbb{C},\\
&\hspace*{3cm} (\mathbf{j},\mathbf{m})\hspace*{1.45cm} \rightarrow\quad \mathcal{B}(\mathbf{j},\mathbf{m})=\int_{\partial\Omega} \mathbf{j}\cdot (\mathbf{m}\wedge \nu)\ ds
\end{split}
\end{equation}
defines a non-degenerate duality product on $TH_{\text{Div}}^{-1/2}(\partial\Omega)$ (cf. \cite{CosLou}). Subtracting \eqref{eq:lem1 3} from \eqref{eq:lem1 2}, one has
\[
\begin{split}
&\int_{D_{\rho/2}} \sigma_a E_\rho\cdot\overline{E}_\rho\ d\sigma_x+\rho^{-2}\int_{D_\rho\backslash D_{\rho/2}}|E_\rho|^2\ d\sigma_x\\
=&\Re \int_{\partial\Omega}\overline{\psi}\cdot [\nu\wedge(\nu\wedge(H_\rho-H_0))]\ ds_x,
\end{split}
\]
which together with the duality \eqref{eq:duality} yields \eqref{eq:estimate 1}.
\end{proof}

In the sequel, we let
\[
\nu\wedge E_\rho^-(x)\ \ \ (\mbox{resp.}\ \ \nu\wedge H_\rho^-(x))\quad \mbox{on\ \ $\partial D_\rho$}
\]
denote the tangential component of $E_\rho$ (resp. $H_\rho$) on $\partial D_\rho$ when one approaches $\partial D_\rho$ from the interior of $D_\rho$. Similarly, we let
\[
\nu\wedge E_\rho^+(x)\ \ \ (\mbox{resp.}\ \ \nu\wedge H_\rho^-(x))\quad \mbox{on\ \ $\partial D_\rho$}
\]
denote the tangential component of $E_\rho$ (resp. $H_\rho$) on $\partial D_\rho$ when one approaches $D_\rho$ from the exterior of $D_\rho$.

\begin{lem}\label{lem:bc control}
The solutions to \eqref{eq:Maxwell free} and \eqref{eq:Maxwell virtual} satisfy
\begin{equation}\label{eq:bc control 1}
\begin{split}
&\left\| (\nu\wedge E_\rho^-) (\rho\ \cdot)\right\|^2_{TH^{-1/2}(\partial D)}\\
\leq & C \rho^{-1}\bigg|1+\omega^2\rho^2\left( 1+i\frac{\rho^{-2}}{\omega} \right )  \bigg|^2 \| \psi \|_{TH^{-1/2}_{\text{\em Div}}(\partial \Omega)} \left\| \nu\wedge (H_\rho-H_0) \right\|_{TH^{-1/2}_{\text{\em Div}}(\partial\Omega)},
\end{split}
\end{equation}
and hence by the transmission condition across $\partial D_\rho$
\begin{equation}\label{eq:bc control 2}
\begin{split}
&\left\| (\nu\wedge E_\rho^+) (\rho\ \cdot)\right\|^2_{TH^{-1/2}(\partial D)}\\
\leq & C \rho^{-1} \bigg|1+\omega^2\rho^2\left( 1+i\frac{\rho^{-2}}{\omega} \right )  \bigg|^2 \| \psi \|_{TH^{-1/2}_{\text{\em Div}}(\partial \Omega)} \left\| \nu\wedge (H_\rho-H_0) \right\|_{TH^{-1/2}_{\text{\em Div}}(\partial\Omega)},
\end{split}
\end{equation}
where $C$ is a positive constant dependent only on $D$ and $\Omega$, but independent of $\psi$ and $\rho$.
\end{lem}

\begin{proof}
We let $E_\rho=(E_\rho^1, E_\rho^2, E_\rho^3)$. 
Clearly, it suffices to show that for $k=1,2,3$,
\begin{equation}\label{eq:bc control 3}
\begin{split}
&\left\|E_\rho^k (\rho\ \cdot)\right\|^2_{H^{-1/2}(\partial D)}\\
\leq & C \rho^{-1}\bigg|1+\omega^2\rho^2\left( 1+i\frac{\rho^{-2}}{\omega} \right )  \bigg|^2 \| \psi \|_{TH^{-1/2}_{\text{ Div}}(\partial \Omega)} \left\| \nu\wedge (H_\rho-H_0) \right\|_{TH^{-1/2}_{\text{ Div}}(\partial\Omega)}\;.
\end{split}
\end{equation}
Our proof begins with the following duality identity
\begin{equation}\label{eq:duality}
\begin{split}
\left\| E_\rho^k(\rho\ \cdot) \right\|_{H^{-1/2}(\partial D)}
=\sup_{\|\phi\|_{H^{1/2}(\partial\Omega)}\leq 1}\bigg| \int_{\partial D} E_\rho^k(\rho x)\cdot \phi(x)\ ds_x \bigg|.
\end{split}
\end{equation}
For any $\phi\in H^{1/2}(\partial D)$, there exists $u\in H^2(D)$ such that (see Theorem 14.1 in \cite{Wlo})\\

\begin{enumerate}
\item[(i)]~~$u=0$ on $\partial D$,\\
\item[(ii)]~~$\frac{\partial u}{\partial \nu}=\phi$ on $\partial D$,\\
\item[(iii)]~~$\|u\|_{H^2(D)}\leq C \|\phi\|_{H^{1/2}(\partial D)}$,\\
\item[(iv)]~~$u=0$ in $D_{1/2}$.
\end{enumerate}
Then
\begin{equation}\label{eq:bc 4}
\int_{\partial D} E_\rho^k(\rho x)\cdot \phi(x)\ ds_x=\int_{\partial D} E_\rho^k(\rho x)\cdot \frac{\partial u(x)}{\partial \nu(x)}\ ds_x.
\end{equation}
For $y\in D_\rho$, let
\[
x:=\frac{y}{\rho}\in D.
\]
Set
\[
E(x):=E_\rho(\rho x)=E_\rho(y),\quad x\in D.
\]
Since
\begin{equation}\label{eq:bc 5}
\begin{split}
 \nabla_y\wedge E_\rho(y)-i\omega H_\rho(y)&=0,\\
\nabla_y\wedge H_\rho(y)+i\omega\left( 1+i\frac{\rho^{-2}}{\omega} \right ) E_\rho(y)&=0,
\end{split}
\end{equation}
for $y\in D_\rho\backslash\overline{D}_{\rho/2}$, it is easily verified that
\begin{equation}\label{eq:bc 6}
\begin{split}
& \nabla_x\wedge E(x)=i\omega \rho H(x),\\
& \nabla_x\wedge H(x)=-i\omega\rho\left( 1+i\frac{\rho^{-2}}{\omega}\right ) E(x),\quad x\in D\backslash\overline{D}_{1/2}.
\end{split}
\end{equation}
Then, by \eqref{eq:bc 4}--\eqref{eq:bc 6}, and Green's formula, we have
\begin{equation}\label{eq:bc 7}
\begin{split}
&\int_{\partial D} E_\rho^k(\rho x)\cdot\phi(x)\ ds_x\\
=&\int_{\partial D} E^k(x)\cdot\frac{\partial u}{\partial \nu}(x)\ ds_x\\
=&\int_{\partial D} E^k(x)\cdot\frac{\partial u}{\partial\nu}(x)-\frac{\partial E^k}{\partial\nu}(x)\cdot u(x)\ ds_x\\
=&\int_{D} E^k(x)\cdot \Delta u(x)-\Delta E^k(x)\cdot u(x)\ d\sigma_x.
\end{split}
\end{equation}
By \eqref{eq:bc 6}, it is straightforward to show that
\begin{equation}\label{eq:bc 8}
\Delta E(x)+\omega^2\rho^2\left(1+i\frac{\rho^{-2}}{\omega} \right ) E(x)=0,\quad x\in D\backslash\overline{D}_{1/2}.
\end{equation}
Therefore, from \eqref{eq:bc 7} and \eqref{eq:bc 8}, it follows directly that
\begin{equation}\label{eq:bc 9}
\begin{split}
& \int_{\partial D} E_\rho^k(\rho x)\cdot \phi(x)\ ds_x\\
=&\int_{D} E^k\cdot \Delta u-u\cdot \Delta E^k\ d\sigma_x\\
=&\left[ 1+\omega^2\rho^2\left( 1+i\frac{\rho^{-2}}{\omega} \right ) \right ]\int_{D\backslash D_{1/2}} E^k\cdot (u+\Delta u)\ d\sigma_x,
\end{split}
\end{equation}
hence
\begin{equation}\label{eq:bc 10}
\begin{split}
& \bigg| \int_{\partial D} E_\rho^k(\rho x)\cdot \phi(x) \ ds_x\bigg |\\
\leq & \bigg| 1+\omega^2\rho^2\left( 1+i\frac{\rho^{-2}}{\omega} \right ) \bigg | \|E\|_{L^2(D\backslash D_{1/2})} \|\phi\|_{H^{1/2}(\partial D)^3}
\end{split}
\end{equation}
Using the relation
\[
\| E \|_{L^2(D\backslash D_{1/2})}=\| E_\rho(\rho\ \cdot) \|_{L^2(D\backslash D_{1/2})}=\rho^{-3/2}\| E_\rho \|_{L^2(D_\rho\backslash D_{\rho/2})},
\]
we have from \eqref{eq:bc 10} that
\[
\begin{split}
& \| E_\rho^k(\rho\ \cdot) \|_{H^{-1/2}(\partial D)}\\
\leq & \bigg| 1+\omega^2\rho^2\left( 1+i\frac{\rho^{-2}}{\omega} \right)  \bigg| \| E\|_{L^2(D\backslash D_{1/2})}\\
\leq & \rho^{-3/2} \bigg| 1+\omega^2\rho^2\left( 1+i\frac{\rho^{-2}}{\omega} \right)  \bigg| \|E_\rho\|_{L^2(D_\rho\backslash D_{\rho/2})},
\end{split}
\]
which together with \eqref{eq:estimate 1} in Lemma~\ref{lem:1} immediately implies \eqref{eq:bc control 3}.

The proof is now completed.

\end{proof}

%
%

The next lemma is of crucial importance and its proof will be given in Section~3.

\begin{lem}\label{lem:crucial}
Suppose $\omega$ is not an eigenvalue of the free-space Maxwell equations \eqref{eq:Maxwell free}. Let $E_0\in H(\nabla\wedge;\Omega)$ and $H_0\in H(\nabla\wedge;\Omega)$ be the solutions to \eqref{eq:Maxwell free}.  Let $\tau\in\mathbb{R}_+$ and let
\[
\varphi\in TH_{\text{\em Div}}^{-1/2}(\partial D_\tau).
\]
Consider the Maxwell equations
\begin{equation}\label{eq:auxiliary}
\begin{cases}
& \nabla\wedge E_\tau-i\omega H_\tau=0\qquad \mbox{in\ \ $\Omega\backslash\overline{D}_\tau$},\\
& \nabla\wedge H_\tau+i\omega E_\tau=0\qquad \mbox{in\ \ $\Omega\backslash\overline{D}_\tau$},\\
& \nu\wedge E_\tau=\varphi\quad\mbox{on\ \ $\partial D_\tau$},\\
& \nu\wedge E_\tau=\psi\quad \mbox{on\ \ $\partial\Omega$}.
\end{cases}
\end{equation}
Then there exists a constant $\tau_0\in\mathbb{R}_+$ such that for any $\tau<\tau_0$,
\begin{equation}\label{eq:crucial}
\begin{split}
&\left\| \nu\wedge(H_\tau-H_0) \right\|_{TH^{-1/2}(\partial\Omega)}\\
\leq & C\left( \tau^3\|\psi\|_{TH^{-1/2}_{\text{\em Div}}(\partial\Omega)} + \tau^2\|\varphi(\tau\ \cdot)\|_{TH^{-1/2}(\partial D)} \right ),
\end{split}
\end{equation}
where $C$ is a generic positive constant dependent only on $\tau_0, \omega$ and $\Omega$, $D$, but independent of $\tau$ and $\varphi$, $\psi$.
\end{lem}

It should be emphasized that in the estimate \eqref{eq:crucial}, the norm for $\varphi(\tau\ \cdot)$ is $TH^{-1/2}(\partial D)$ though $\varphi(\tau\ \cdot)\in TH^{-1/2}_{\text{Div}}(\partial D)$, and in this sense, the estimate is ``non-standard".

We are in position to present the proof of Theorem~\ref{thm:main}.

\begin{proof}[Proof of Theorem~\ref{thm:main}]

By taking $\tau=\rho$ and $\varphi=\nu\wedge E_\rho^+|_{\partial D_\rho}$ in Lemma~\ref{lem:crucial}, we have
\begin{equation}\label{eq:m1}
\begin{split}
& \|\nu\wedge(H_\rho-H_0) \|_{TH^{-1/2}_{\text{Div}}(\partial \Omega)}\\
\leq & C_1\left( \rho^3\|\psi\|_{TH_{\text{Div}}}^{-1/2}(\partial\Omega)+\rho^2\| (\nu\wedge E_\rho^+)(\rho\ \cdot) \|_{TH^{-1/2}(\partial D)} \right ).
\end{split}
\end{equation}
Next, by Lemma~\ref{lem:bc control}, we have for $\epsilon>0$
\begin{equation}\label{eq:m2}
\begin{split}
&\|(\nu\wedge E_\rho^+)(\rho\ \cdot ) \|_{TH^{-1/2}(\partial D)}\\
\leq & C_2 \rho^{-1/2} \|\psi\|^{1/2}_{TH_{\text{Div}}^{-1/2}(\partial\Omega)} \|\nu\wedge (H_\rho-H_0)\|^{1/2}_{TH^{-1/2}_{\text{Div}}(\partial\Omega)}\\
\leq & C_2 \rho^{-1/2} \left(  \frac{\rho^{3/2}}{4\epsilon}\| \psi \|_{TH^{-1/2}_{\text{Div}}(\partial\Omega)}+ \frac{\epsilon}{\rho^{3/2}}\|\nu\wedge (H_\rho-H_0)\|_{TH_{\text{Div}}^{-1/2}(\partial\Omega)} \right ).
\end{split}
\end{equation}
From \eqref{eq:m1} and \eqref{eq:m2}, we further have
\begin{equation}\label{eq:m3}
\begin{split}
& \| \nu\wedge (H_\rho-H_0) \|_{TH^{-1/2}_{\text{Div}}(\partial\Omega)}\leq C_1\rho^3 \|\psi\|_{TH^{-1/2}_{\text{Div}}(\partial\Omega)}\\
& +\frac{1}{4} C_1 C_2 \frac{\rho^3}{\epsilon} \|\psi\|_{TH^{-1/2}_{\text{Div}}(\partial\Omega)}+C_1C_2\epsilon \|\nu\wedge (H_\rho-H_0)\|_{TH^{-1/2}_{\text{Div}}(\partial\Omega)}.
\end{split}
\end{equation}
By choosing $\epsilon$ such that $C_1C_2\epsilon<1/2$, we see immediately from \eqref{eq:m3} that
\[
\| \nu\wedge (H_\rho-H_0) \|_{TH^{-1/2}_{\text{Div}}(\partial\Omega)}\leq C\rho^3 \|\psi\|_{TH^{-1/2}_{\text{Div}}(\partial\Omega)},
\]
which completes the proof.

\end{proof}

\section{Proof of Lemma~\ref{lem:crucial}}

We present the proof of Lemma~\ref{lem:crucial} which is crucial for the proof of our main theorem.

Set
\[
\widetilde{E}_\tau=E_\tau-E_0,\qquad \widetilde{H}_\tau=H_\tau-H_0.
\]
It is straightforward to verify that
\begin{equation}\label{eq:l31}
\begin{cases}
& \nabla\wedge \widetilde{E}_\tau-i\omega \widetilde{H}_\tau=0\qquad\, \mbox{in\ \ $\Omega\backslash\overline{D}_\tau$},\\
& \nabla\wedge \widetilde{H}_\tau+i\omega \widetilde{E}_\tau=0\qquad\, \mbox{in\ \ $\Omega\backslash\overline{D}_\tau$},\\
& \nu\wedge \widetilde{E}_\tau=\varphi-\nu\wedge E_0\hspace*{0.64cm}\mbox{on\ \ $\partial D_\tau$},\\
& \nu\wedge \widetilde{E}_\tau=0\hspace*{2.27cm} \mbox{on\ \ $\partial\Omega$}.
\end{cases}
\end{equation}
Obviously, in order to show \eqref{eq:crucial}, it suffices to show
\begin{equation}\label{eq:l32}
\|\nu\wedge\widetilde{H}_\tau\|_{TH^{-1/2}_{\text{Div}}(\partial\Omega)} \leq C\left( \tau^3\|\psi\|_{TH^{-1/2}_{\text{ Div}}(\partial\Omega)} + \tau^2\|\varphi(\tau\ \cdot)\|_{TH^{-1/2}(\partial D)} \right ).
\end{equation}
In order to prove \eqref{eq:l32}, we let
\begin{equation}\label{eq:split1}
\widetilde{E}_\tau=U_\tau-\widetilde{U}_\tau\qquad\mbox{and}\qquad \widetilde{H}_\tau=V_\tau-\widetilde{V}_\tau,
\end{equation}
where $(U_\tau, V_\tau)\in H_{loc}(\nabla\wedge; \mathbb{R}^3\backslash\overline{D}_\tau)\wedge H_{loc}(\nabla\wedge; \mathbb{R}^3\backslash\overline{D}_\tau)$ are scattering solutions to
\begin{equation}\label{eq:whole}
\begin{split}
& \nabla\wedge U_\tau-i\omega V_\tau=0\qquad\, \ \ \mbox{in\ \ $\mathbb{R}^3\backslash\overline{D}_\tau$},\\
& \nabla\wedge V_\tau+i\omega U_\tau=0\qquad \ \ \, \mbox{in\ \ $\mathbb{R}^3\backslash\overline{D}_\tau$},\\
& \nu\wedge U_\tau=\varphi-\nu\wedge E_0\qquad \mbox{on\ \ $\partial D_\tau$},\\
\lim_{|x|\rightarrow+\infty} & |x|\left| (\nabla\wedge U_\tau)(x)\wedge\frac{x}{|x|}-i\omega U_\tau(x) \right|=0,
\end{split}
\end{equation}
and $(\widetilde{U}_\tau, \widetilde{V}_\tau)\in H(\nabla\wedge;\Omega\backslash\overline{D}_\tau)\wedge H(\nabla\wedge; \Omega\backslash\overline{D}_\tau)$ are solutions to
\begin{equation}\label{eq:bdomain}
\begin{split}
& \nabla\wedge \widetilde{U}_\tau-i\omega \widetilde{V}_\tau=0\qquad\ \mbox{in\ \ $\Omega\backslash\overline{D}_\tau$},\\
& \nabla\wedge\widetilde{V}_\tau+i\omega\widetilde{U}_\tau=0\qquad\ \mbox{in\ \ $\Omega\backslash\overline{D}_\tau$},\\
& \nu\wedge \widetilde{U}_\tau=\nu\wedge U_\tau\qquad\quad\ \, \mbox{on\ \ $\partial\Omega$},\\
& \nu\wedge\widetilde{U}_\tau=0\qquad\qquad\ \quad\ \, \mbox{on\ \ $\partial D_\tau$}.
\end{split}
\end{equation}

We shall show the following two lemmas, which immediately imply \eqref{eq:l32}

\begin{lem}\label{lem:s1}
Let $(U_\tau, V_\tau)\in H_{loc}(\nabla\wedge; \mathbb{R}^3\backslash\overline{D}_\tau)\wedge H_{loc}(\nabla\wedge; \mathbb{R}^3\backslash\overline{D}_\tau)$ be scattering solutions to \eqref{eq:whole}. Then there exists $\tau_0\in\mathbb{R}_+$ such that for any $\tau<\tau_0$
\begin{equation}\label{eq:s11}
\|U_\tau\|_{(C^2(\partial\Omega))^3} \leq C\left( \tau^3\|\psi\|_{TH^{-1/2}_{\text{\em Div}}(\partial\Omega)} + \tau^2\|\varphi(\tau\ \cdot)\|_{TH^{-1/2}(\partial D)} \right ),
\end{equation}
and
\begin{equation}\label{eq:s12}
\|V_\tau\|_{(C^2(\partial\Omega))^3} \leq C\left( \tau^3\|\psi\|_{TH^{-1/2}_{\text{\em Div}}(\partial\Omega)} + \tau^2\|\varphi(\tau\ \cdot)\|_{TH^{-1/2}(\partial D)} \right ),
\end{equation}
where $C$ is a positive constant dependent only on $\tau_0, \omega$ and $\Omega$, $D$, but independent of $\tau$ and $\varphi$, $\psi$.
\end{lem}

\begin{lem}\label{lem:s2}
Let $(\widetilde{U}_\tau, \widetilde{V}_\tau)\in H(\nabla\wedge; \Omega\backslash\overline{D}_\tau)\wedge H(\nabla\wedge; \Omega\backslash\overline{D}_\tau)$ be solutions to \eqref{eq:bdomain}. Then there exists $\tau_0\in\mathbb{R}_+$ such that for any $\tau<\tau_0$
\begin{equation}\label{eq:s21}
\|\nu\wedge \widetilde{V}_\tau\|_{TH_{\text{\em Div}}^{-1/2}(\partial\Omega)} \leq C\left( \tau^3\|\psi\|_{TH^{-1/2}_{\text{\em Div}}(\partial\Omega)} + \tau^2\|\varphi(\tau\ \cdot)\|_{TH^{-1/2}(\partial D)} \right ),
\end{equation}
where $C$ is a positive constant dependent only on $\tau_0, \omega$ and $\Omega$, $D$, but independent of $\tau$ and $\varphi$, $\psi$.
\end{lem}

We next present the proof of Lemma~\ref{lem:s1}, which may be further divided into the following two propositions.

\begin{prop}\label{lem:s11}
Let $(U_{\tau,1}, V_{\tau, 1})\in H_{loc}(\nabla\wedge; \mathbb{R}^3\backslash\overline{D}_\tau)\wedge H_{loc}(\nabla\wedge; \mathbb{R}^3\backslash\overline{D}_\tau)$ be solutions to
\begin{equation}\label{eq:whole 1}
\begin{split}
& \nabla\wedge U_{\tau,1}-i\omega V_{\tau,1}=0\qquad\, \ \ \mbox{in\ \ $\mathbb{R}^3\backslash\overline{D}_\tau$},\\
& \nabla\wedge V_{\tau,1}+i\omega U_{\tau,1}=0\qquad \ \ \, \mbox{in\ \ $\mathbb{R}^3\backslash\overline{D}_\tau$},\\
& \nu\wedge U_{\tau,1}=\nu\wedge E_0\hspace*{1.8cm} \mbox{on\ \ $\partial D_\tau$},\\
\lim_{|x|\rightarrow+\infty} & |x|\left| (\nabla\wedge U_{\tau,1})(x)\wedge\frac{x}{|x|}-i\omega U_{\tau,1}(x) \right|=0,
\end{split}
\end{equation}
Then there exists $\tau_0\in\mathbb{R}_+$ such that for any $\tau<\tau_0$
\begin{equation}\label{eq:s111}
\|U_{\tau,1}\|_{(C^2(\partial\Omega))^3} \leq C \tau^3\|\psi\|_{TH^{-1/2}_{\text{\em Div}}(\partial\Omega)}
\end{equation}
and
\begin{equation}\label{eq:s122}
\|V_{\tau,1}\|_{(C^2(\partial\Omega))^3} \leq C \tau^3\|\psi\|_{TH^{-1/2}_{\text{\em Div}}(\partial\Omega)} ,
\end{equation}
where $C$ is a positive constant dependent only on $\tau_0, \omega$ and $\Omega$, $D$, but independent of $\tau$ and $\varphi$, $\psi$.
\end{prop}

\begin{prop}\label{lem:s12}
Let $(U_{\tau,2}, V_{\tau, 2})\in H_{loc}(\nabla\wedge; \mathbb{R}^3\backslash\overline{D}_\tau)\wedge H_{loc}(\nabla\wedge; \mathbb{R}^3\backslash\overline{D}_\tau)$ be solutions to
\begin{equation}\label{eq:whole 1}
\begin{split}
& \nabla\wedge U_{\tau,2}-i\omega V_{\tau,2}=0\qquad\, \ \ \mbox{in\ \ $\mathbb{R}^3\backslash\overline{D}_\tau$},\\
& \nabla\wedge V_{\tau,2}+i\omega U_{\tau,2}=0\qquad \ \ \, \mbox{in\ \ $\mathbb{R}^3\backslash\overline{D}_\tau$},\\
& \nu\wedge U_{\tau,2}=\varphi \hspace*{2.6cm} \mbox{on\ \ $\partial D_\tau$},\\
\lim_{|x|\rightarrow+\infty} & |x|\left| (\nabla\wedge U_{\tau,2})(x)\wedge\frac{x}{|x|}-i\omega U_{\tau,2}(x) \right|=0,
\end{split}
\end{equation}
Then there exists $\tau_0\in\mathbb{R}_+$ such that for any $\tau<\tau_0$
\begin{equation}\label{eq:s211}
\|U_{\tau,2}\|_{(C^2(\partial\Omega))^3} \leq C \tau^2\|\varphi(\tau\ \cdot)\|_{TH^{-1/2}(\partial D)}
\end{equation}
and
\begin{equation}\label{eq:s222}
\|V_{\tau,2}\|_{(C^2(\partial\Omega))^3} \leq C \tau^2\|\varphi(\tau\ \cdot)\|_{TH^{-1/2}(\partial D)} ,
\end{equation}
where $C$ is a positive constant dependent only on $\tau_0, \omega$ and $\Omega$, $D$, but independent of $\tau$ and $\varphi$, $\psi$.
\end{prop}

\begin{proof}[Proof of Proposition~\ref{lem:s11}]

We first note that the solutions $E_0$ and $H_0$ to \eqref{eq:Maxwell free} are smooth inside $\Omega$, and also by the local regularity estimate we have
\begin{equation}\label{eq:loca estimate 1}
\|E_0\|_{(C^1(\overline{D}))^3}\leq C \| \psi \|_{TH^{-1/2}_{\text{Div}}(\partial\Omega)}\ \mbox{and}\ \|H_0\|_{(C^1(\overline{D}))^3}\leq C \| \psi \|_{TH^{-1/2}_{\text{Div}}(\partial\Omega)},
\end{equation}
where $C$ is constant depending only on $\Omega, D$ and $\omega$. Since $\nu\wedge E_0$ is smooth on $\partial D_\tau$, we know both $U_{\tau,1}$ and $V_{\tau,1}$ are strong solutions which belong to $(C^1(\mathbb{R}^3\backslash\overline{D}_\tau)\cap C^{0,\alpha}(\mathbb{R}^3\backslash D_\tau))^3$ with $0<\alpha<1$ (see \cite{ColKre1, ColKre}). By a completely similar argument to the proof of Corollary 3.2 in \cite{APRT}, which is based on the low frequency asymptotics in \cite{Kle}, one can show \eqref{eq:s111} and \eqref{eq:s122}.

\end{proof}

\begin{proof}[Proof of Proposition~\ref{lem:s12}]
We make use of the layer potential technique to show the lemma. To that end, we let
\[
G(x,y):=\frac{1}{4\pi}\frac{e^{i\omega |x-y|}}{|x-y|}\quad \mbox{and}\quad G_0(x,y):=\frac{1}{4\pi}\frac{1}{|x-y|},\ x, y\in\mathbb{R}^3; \ x\neq y.
\]
Furthermore, we introduce the following vector boundary layer potential operators $\mathbf{M}_{\Gamma}$ and $\mathbf{M}_{\Gamma}^0$,
\begin{equation}\label{eq:linear operator}
(\mathbf{M}_\Gamma \mathbf{a})(x):=2\int_\Gamma \nu(x)\wedge [\nabla_x\wedge(\mathbf{a}(y) G(x,y)]\ ds_y,\quad x\in\Gamma
\end{equation}
and
\begin{equation}\label{eq:linear operator zero}
(\mathbf{M}^0_\Gamma \mathbf{a})(x):=2\int_\Gamma \nu(x)\wedge [\nabla_x\wedge(\mathbf{a}(y) G_0(x,y)]\ ds_y,\quad x\in\Gamma,
\end{equation}
where $\mathbf{a}$ is a tangential vector field on $\Gamma$. We refer to \cite{ColKre1,Mcl,Ned,Tay} for related mapping properties of the above introduced operators.

We make use of the following {\it ansatz} of the EM fields to \eqref{eq:whole 1},
\begin{align}
U_{\tau,2}(x)=&\nabla_x\wedge\int_{\partial D_\tau} G(x,y)\mathbf{a}(y)\ ds_y,\ \ x\in\mathbb{R}^3\backslash\overline{D}_\tau,\label{eq:U2}\\
V_{\tau,2}(x)=&\frac{1}{i\omega}\nabla_x\wedge U_{\tau,2}(x)
=\frac\omega i\int_{\partial D_\tau} G(x,y)\mathbf{a}(y)\ ds_y\nonumber\\
&+\frac{1}{i\omega}\int_{\partial D_\tau} \nabla_x G(x,y) \text{Div}\, \mathbf{a}(y)\ ds_y,\ \ x\in\mathbb{R}^3\backslash\overline{D}_\tau,\label{eq:V2}
\end{align}
where $\mathbf{a}\in TH_{\text{Div}}^{-1/2}(\partial D_\tau)$. One needs first show the uniform well-posedness of the Maxwell equations \eqref{eq:whole 1}. That is, for any fixed $\omega\in\mathbb{R}_+$, there exists $\tau_0\in\mathbb{R}_+$ sufficiently small such that when $\tau<\tau_0$, there exists a unique $\mathbf{a}\in TH_{\text{Div}}^{-1/2}(\partial D_\tau)$ such that $U_{\tau,2}$ and $V_{\tau,2}$ given in \eqref{eq:U2} and \eqref{eq:V2} are solutions to \eqref{eq:whole 1}. However, we shall not show this in the sequel, and instead we shall directly estimate $\mathbf{a}$ assuming its existence which will then give the desired estimates \eqref{eq:s211} and \eqref{eq:s222}. Nevertheless, we emphasize that it can be directly seen from our estimating of $\mathbf{a}$ the uniform well-posedness of \eqref{eq:whole 1}, or equivalently the unique existence of $\mathbf{a}\in TH_{\text{Div}}^{-1/2}(\partial D_\tau)$.

Next, by letting $x$ approach $\partial D_\tau^+$, and using the mapping properties of $\mathbf{M}_{\partial D_\tau}$ and the jump properties of the vector potential operator $\mathbf{M}_{\partial D_\tau}$, one has
\begin{equation}\label{eq:integral equation 1}
\mathbf{a}(x)+[\mathbf{M}_{\partial D_\tau}\mathbf{a}](x)=2\varphi(x),\quad x\in\partial D_\tau.
\end{equation}
We claim that for $\tau$ sufficiently small
\begin{equation}\label{eq:claim1}
\|\mathbf{a}(\tau\ \cdot)\|_{TH^{-1/2}(\partial D)}\leq C \|\varphi(\tau\ \cdot)\|_{TH^{-1/2}(\partial D)},
\end{equation}
where $C$ is a generic constant independent of $\tau$ and $\varphi$. To that end, we let
\[
\widetilde{\mathbf{a}}(x')=:\mathbf{a}(\tau x'), \quad  x':=\frac x \tau\in\partial D,\ \ x\in\partial D_\tau.
\]
By using change of variables,
\begin{equation}\label{eq:change variable}
\mbox{ $x'=x/\tau$ and $y'=y/\tau$ for $x, y\in\partial D_\tau$,}
\end{equation}
one can show
\begin{equation}\label{eq:change1}
\begin{split}
& (\mathbf{M}_{\partial D_\tau}\mathbf{a})(x)=(\mathbf{M}_{\partial D_\tau}\mathbf{a})(\tau x')\\
=& 2 \nu(x')\wedge \nabla_{x'}\wedge\int_{\partial D} G_{\tau}(x',y')\widetilde{\mathbf{a}}(y')\ ds_{y'},
\end{split}
\end{equation}
where
\[
G_{\tau}(x',y')=\frac{1}{4\pi}\frac{e^{i\tau\omega|x'-y'|}}{|x'-y'|}=G_0(x',y')+\frac{i\tau\omega}{4\pi}+\tau^2\mathscr{R}(x',y').
\]
It is easily verified that the remainder term $\mathscr{R}(x',y')$ satisfies
\begin{equation}\label{eq:remainder}
|\mathscr{R}(x',y')|=\mathcal{O}(|x'-y'|),\quad x', y'\in\partial D.
\end{equation}
Hence, we have the following splitting
\begin{equation}\label{eq:splitting}
\begin{split}
& 2\nu(x')\wedge\nabla_{x'}\wedge\int_{\partial D} G_{\tau}(x',y')\widetilde{\mathbf{a}}(y')\ ds_{y'}\\
=& 2 \nu(x')\wedge\nabla_{x'}\wedge \int_{\partial D} G_0(x',y')\widetilde{\mathbf{a}}(y')\ ds_y'\\
&+2\tau^2\nu(x')\wedge\nabla_{x'}\wedge\int_{\partial D} \mathscr{R}(x',y')\widetilde{\mathbf{a}}(y')\ ds_{y'}\\
=& (\mathbf{M}_{\partial D}^0\widetilde{\mathbf{a}})(x')+(\mathcal{R}\widetilde{\mathbf{a}})(x').
\end{split}
\end{equation}
By using the mapping properties of layer potential operators in \cite{Ned}, it is straightforward to show that
\begin{equation}\label{eq:remainder 2}
\begin{split}
& \|\mathcal{R}\widetilde{\mathbf{a}}(\cdot)\|_{TH^{-1/2}(\partial D)}\\
\leq & C \tau^2 \|\widetilde{\mathbf{a}}(\cdot)\|_{TH^{-1/2}(\partial D)}=
C\tau^2\|\mathbf{a}(\tau\ \cdot)\|_{TH^{-1/2}(\partial D)},
\end{split}
\end{equation}
where $C$ is a generic constant depending only on $D$ and $\omega$.
Then, using again the change of variables in \eqref{eq:change variable} to the integral equation
\eqref{eq:integral equation 1}, and the splitting \eqref{eq:splitting}, one has by direct calculations that
\begin{equation}\label{eq:change 3}
[I+M_{\partial D}^0+\mathcal{R}]\widetilde{\mathbf{a}}(x')=2\varphi(\tau x'),\quad x'\in\partial D.
\end{equation}
Next, we shall show
\begin{equation}\label{eq:invertible}
\mbox{$I+M_{\partial D}^0$ is invertible from $TH^{-1/2}(\partial D)$ to $TH^{-1/2}(\partial D)$},
\end{equation}
which together with \eqref{eq:remainder} and \eqref{eq:change 3} that
\begin{equation}\label{eq:invertible 2}
\widetilde{\mathbf{a}}(x')=2[(I+M_{\partial D}^0)^{-1}+\mathcal{O}(\tau^2)](\varphi(\tau\ \cdot))(x'),
\end{equation}
thus proving the claim in \eqref{eq:claim1}. In order to show \eqref{eq:invertible}, we first note that $M_{\partial D}^0$
is an integral operator of order $-1$, i.e., it is continuous from $TH^s(\partial D)$ to $TH^{s+1}(\partial D)$ (see \cite{Ned}, pp. 242). Hence, by the Reisz-Fredholm theory, it is sufficient to show that
\begin{equation}\label{eq:Fred}
(I+M_{\partial D}^0)\mathbf{b}=0,\quad \mathbf{b}\in TH^{-1/2}(\partial D)
\end{equation}
has only trivial solution, i.e., $\mathbf{b}=0$. By using the fact that $M_{\partial D}^0$ is of degree $-1$, we see that the spectrum of $M_{\partial D}^0$ is the same in $TH^{-1/2}(\partial D)$ and $C(\partial D)$. Then by Theorem~5.4 in \cite{ColKre1}, one must have $\mathbf{b}=0$ in \eqref{eq:Fred}, which readily proves \eqref{eq:invertible}.

Finally, using \eqref{eq:claim1} and the integral representations \eqref{eq:U2} and \eqref{eq:V2}, it is by direct calculations to show \eqref{eq:s211} and \eqref{eq:s222}.

The proof is completed.
\end{proof}

\begin{proof}[Proof of Lemma~\ref{lem:s2}]
Since $U_\tau$ is smooth on $\partial \Omega$, we know both $\widetilde{U}_{\tau}$ and $\widetilde{V}_{\tau}$ are strong solutions which belong to $(C^1(\Omega\backslash\overline{D}_\tau)\cap C^{0,\alpha}(\overline{\Omega}\backslash D_\tau))^3$ with $0<\alpha<1$ (see \cite{ColKre1, ColKre}). Hence, in the sequel, we shall work within the classic setting. To that end, we introduce $T(\Gamma)$, the spaces of all continuous tangential fields $\mathbf{a}$ equipped with the supremum norm, and $T^{0,\alpha}(\Gamma)$, the space of all H\"older continuous tangential fields equipped with the usual H\"older norm. We also need to introduce normed spaces of tangential fields possessing a surface divergence by
\[
T_d(\Gamma):=\{\mathbf{a}\in T(\Gamma)| \text{Div}\, \mathbf{a}\in C(\Gamma)\}
\]
and
\[
T_d^{0,\alpha}(\Gamma):=\{\mathbf{a}\in T^{0,\alpha}| \text{Div}\, \mathbf{a}\in C^{0,\alpha}(\Gamma)\}
\]
equipped with the norms
\[
\|\mathbf{a}\|_{T_d}:=\|\mathbf{a}\|_\infty+\|\text{Div}\, \mathbf{a}\|_\infty,\quad \|a\|_{T_d^{0,\alpha}}:=\|a\|_{0,\alpha}+\|\text{Div}\, \mathbf{a}\|_{0,\alpha}.
\]
We again employ the boundary layer potential operators introduced in \eqref{eq:linear operator} and \eqref{eq:linear operator zero}, and refer to \cite{ColKre1} and \cite{ColKre} for mapping and jumping properties in the classical setting.

Similar to the proof of Proposition~\ref{lem:s12}, instead of proving the uniform well-posedness of the Maxwell equations \eqref{eq:bdomain}, we focus on deriving the desired estimate \eqref{eq:s21}. However, it should be pointed out that the unique existence of strong solutions to \eqref{eq:bdomain} for sufficiently small $\tau$ can be directly seen from our subsequent argument.

By the Stratton-Chu formula, we have (see \cite{ColKre}, Theorem 6.2)
\begin{equation}\label{eq:int representation n1}
\begin{split}
\widetilde{V}_{\tau}(x)&= -\nabla_x\wedge\int_{\partial \Omega} \nu(y)\wedge \widetilde{V}_{\tau}(y) G(x,y)\ ds_y\\
&+\nabla_x\wedge\int_{\partial D_\tau}\nu(y)\wedge\widetilde{V}_\tau(y)G(x,y)\ ds_y\\
&-\frac{1}{i\omega}\nabla_x\wedge\nabla_x\wedge\int_{\partial \Omega} \nu(y)\wedge U_{\tau}(y) G(x,y)\ ds_y,
\qquad x\in\Omega\backslash\overline{D}_\tau.
\end{split}
\end{equation}
Set
\begin{align*}
\mathbf{a}_1(x)=&\nu(x)\wedge \widetilde{V}_{\tau}(x),\quad x\in\partial \Omega,\\
\mathbf{a}_2(x)=&\nu(x)\wedge \widetilde{V}_{\tau}(x),\quad x\in\partial D_\tau,
\end{align*}
and
\begin{align*}
P_1(x)=& -\frac{1}{i\omega}\nu(x)\wedge\nabla_x\wedge\nabla_x\wedge\int_{\partial \Omega} \nu(y)\wedge U_{\tau}(y) G(x,y)\ ds_y,
\qquad x\in\partial\Omega,\\
P_2(x)=& -\frac{1}{i\omega}\nu(x)\wedge\nabla_x\wedge\nabla_x\wedge\int_{\partial \Omega} \nu(y)\wedge U_{\tau}(y) G(x,y)\ ds_y,
\qquad x\in\partial D_\tau
\end{align*}
By letting $x$ approach $\partial D_\tau^+$, and using the mapping properties of $\mathbf{M}_{\partial D_\tau}$ and $\mathbf{M}_{\partial\Omega}$, one has
\begin{equation}\label{eq:integral system}
\begin{split}
(I+\mathbf{M}_{\partial\Omega})\mathbf{a}_1(x)-(\mathbf{M}_{\partial D_\tau}\mathbf{a}_2)(x)=& 2P_1(x),\quad x\in\partial\Omega,\\
(I-\mathbf{M}_{\partial D_\tau})\mathbf{a}_2(x)+(\mathbf{M}_{\partial\Omega}\mathbf{a}_1)(x)=& 2 P_2(x),\quad x\in\partial D_\tau .
\end{split}
\end{equation}
By using a similar scaling argument to that in the proof of Proposition~\ref{lem:s12}, one can show
\begin{equation}\label{eq:scaling}
(\mathbf{M}_{\partial D_\tau}\mathbf{a}_2)(\tau x')=(\mathbf{M}_{\partial D}^0\widetilde{\mathbf{a}}_2)(x')+(\mathcal{R}\widetilde{\mathbf{a}}_2)(x'),\quad x'\in\partial D,
\end{equation}
where $\widetilde{\mathbf{a}}_2(x'):=\mathbf{a}_2(\tau x')$ for $x'\in\partial D$, and
\begin{equation}\label{eq:scale 2}
\|\mathcal{R}\widetilde{\mathbf{a}}_2\|_{T_d^{0,\alpha}(\partial D)}\leq C\tau^2 \|\widetilde{\mathbf{a}}_2\|_{T_d^{0,\alpha}(\partial D)}.
\end{equation}
By setting $\widetilde{P}_2(x')=P_2(\tau x')$ for $x'\in\partial D$, the system of integral equations \eqref{eq:integral system} can be further formulated as
\begin{equation}\label{eq:integral system 2}
\begin{split}
(I+\mathbf{M}_{\partial\Omega})\mathbf{a}_1(x)-(\widetilde{\mathbf{M}}_{\partial D}\widetilde{\mathbf{a}}_2)(x)=& 2 P_1(x),\quad x\in\partial \Omega,\\
(I-M_{\partial D}^0)\widetilde{\mathbf{a}}_2(x')-(\mathcal{R}\widetilde{\mathbf{a}}_2)(x')+(\mathbf{M}_{\partial\Omega}\mathbf{a}_1)(\tau x')=& 2 \widetilde{P}_2(x'),\quad x'\in\partial D,
\end{split}
\end{equation}
where
\[
(\widetilde{\mathbf{M}}_{\partial D}\widetilde{\mathbf{a}}_2)(x)=(\mathbf{M}_{\partial D_\tau}\mathbf{a}_2)(x),\quad x\in\partial\Omega,
\]
obtained by replacing the variable $y$ in defining $\mathbf{M}_{\partial D_\tau}\mathbf{a}_2$ (cf. \eqref{eq:linear operator}) by $\tau y'$ with $y'\in\partial D$. Since the integral kernel for $(\widetilde{\mathbf{M}}_{\partial D}\widetilde{\mathbf{a}}_2)(x)$ is smooth when $x\in\partial\Omega$, it is straightforwardly shown that
\begin{equation}\label{eq:oper asymp 1}
\|\widetilde{\mathbf{M}}_{\partial D}\widetilde{\mathbf{a}}_2\|_{T_d^{0,\alpha}(\partial \Omega)}\leq C \tau^2\|\widetilde{\mathbf{a}}_2\|_{T_d^{0,\alpha}(\partial D)}.
\end{equation}
Similarly, one can see that
\begin{equation}\label{eq:oper asymp 2}
\|\mathbf{M}_{\partial\Omega}\mathbf{a}_1(\tau\ \cdot)\|_{T_{d}^{0,\alpha}(\partial D)}\leq C \| \mathbf{a}_1\|_{T_d^{0,\alpha}(\partial\Omega)}.
\end{equation}

Define
\begin{equation}\label{eq:syst operators 1}
\mathbf{a}=\left( \begin{array}{c}
\mathbf{a}_1\\
\widetilde{\mathbf{a}}_2
\end{array} \right ),\quad \mathbf{P}=\left( \begin{array}{c}
2P_1\\
2\widetilde{P}_2
\end{array} \right ),
\end{equation}
and
\begin{equation}\label{eq:syst operator 2}
\mathbf{L}=\left(\begin{array}{cc}
\mathbf{L}_{11} & \mathbf{0}\\
\mathbf{L}_{21} & \mathbf{L}_{22}
\end{array} \right ),\quad \mathbf{R}=\left(\begin{array}{cc}
\mathbf{0} & \mathbf{R}_{12}\\
\mathbf{0} & \mathbf{R}_{22}
\end{array} \right ),
\end{equation}
where
\[
\mathbf{L}_{11}:=I+\mathbf{M}_{\partial\Omega},\quad \mathbf{L}_{21}:=\mathbf{M}_{\partial\Omega},\quad \mathbf{L}_{22}:=I-M_{\partial D}^0
\]
and
\[
\mathbf{R}_{12}:=\widetilde{\mathbf{M}},\quad \mathbf{R}_{22}:=\mathcal{R}.
\]
Then the system of integral equations \eqref{eq:integral system 2} can be written as
\begin{equation}\label{eq:sysint}
(\mathbf{L}-\mathbf{R})\mathbf{a}=\mathbf{P}.
\end{equation}
$\mathbf{L}_{11}$ is compact from $T_d^{0,\alpha}(\partial \Omega)$ to $T_d^{0,\alpha}(\partial D)$ and $\mathbf{L}_{22}$ is compact from $T_d^{0,\alpha}(\partial D)$ to $T_d^{0,\alpha}(\partial D)$ (cf. \cite{ColKre}, Chapter 6). Since $\omega$ is not an EM eigenvalue to \eqref{eq:Maxwell free}, by the Riesz-Fredholm theory, $\mathbf{L}_{11}$ is invertible (cf. \cite{ColKre1}, Theorem 4.23), and $\mathbf{L}_{22}$ is also invertible (cf. \cite{ColKre1}, Theorem 5.4). Hence, it is straightforward to verify that $\mathbf{L}$ is invertible from $T_d^{0,\alpha}(\partial\Omega)\wedge T_d^{0,\alpha}(\partial D)$ to itself, and
\[
\mathbf{L}^{-1}=\left(\begin{array}{cc}
\mathbf{L}_{11}^{-1} & \mathbf{0}\\
-\mathbf{L}_{22}^{-1}\mathbf{L}_{21}\mathbf{L}_{11}^{-1} & \mathbf{L}_{22}^{-1}
\end{array} \right ).
\]
Hence, by noting \eqref{eq:scale 2} and \eqref{eq:oper asymp 1}, we see that for sufficiently small $\tau$,
\begin{equation}\label{eq:final 1}
\mathbf{a}=(\mathbf{L}-\mathbf{R})^{-1}\mathbf{P}.
\end{equation}
Finally, by using the mapping property of the electric dipole operator (see Theorem 6.17 in \cite{ColKre}), one can show
\[
\|P_1\|_{T_d^{0,\alpha}(\partial \Omega)}\leq C \|U_\tau\|_{(C^2(\partial\Omega))^3},
\]
which together with Proposition~\ref{lem:s12} further implies that
\begin{equation}\label{eq:final 2}
\|P_1\|_{T_d^{0,\alpha}(\partial\Omega)}\leq  C\left( \tau^3\|\psi\|_{TH^{-1/2}_{\text{\em Div}}(\partial\Omega)} + \tau^2\|\varphi(\tau\ \cdot)\|_{TH^{-1/2}(\partial D)} \right ).
\end{equation}
Moreover, it is straightforward to verify that
\[
\|\widetilde{P}_2\|_{T_d^{0,\alpha}(\partial D)}\leq C \|U_\tau\|_{(C^2(\partial\Omega))^3}
\]
and hence
\begin{equation}\label{eq:final 3}
\|\widetilde{P}_2\|_{T_d^{0,\alpha}(\partial D)}\leq C\left( \tau^3\|\psi\|_{TH^{-1/2}_{\text{\em Div}}(\partial\Omega)} + \tau^2\|\varphi(\tau\ \cdot)\|_{TH^{-1/2}(\partial D)} \right ).
\end{equation}
By combining \eqref{eq:final 1}, \eqref{eq:final 2} and \eqref{eq:final 3}, one readily has
\[
\|\mathbf{a}_1\|_{T_d^{0,\alpha}(\partial\Omega)}\leq C\left( \tau^3\|\psi\|_{TH^{-1/2}_{\text{\em Div}}(\partial\Omega)} + \tau^2\|\varphi(\tau\ \cdot)\|_{TH^{-1/2}(\partial D)} \right ),
\]
which immediately implies \eqref{eq:s21}.

The proof of Lemma~\ref{lem:s2} is complete.

\end{proof}

\section*{Acknowledgement}
The research of GB was supported
in part by the NSF grants DMS-0908325, DMS-0968360,
DMS-1211292,  the ONR grant N00014-12-1-0319, a Key
Project of the Major Research Plan of NSFC (No. 91130004), and a
special research grant from Zhejiang University. The research of HL is supported by NSF grant DMS-1207784.

\end{document}